\input ./style/arxiv-vmsta.cfg
\documentclass[numbers,compress,v1.0.1]{vmsta}

\volume{2}
\issue{3}
\pubyear{2015}
\firstpage{219}
\lastpage{232}
\doi{10.15559/15-VMSTA35CNF}

\newtheorem{thm}{Theorem}
\newtheorem{lemma}{Lemma}

\theoremstyle{definition}

\newtheorem{assum}{Assumption}
\newtheorem{remark}{Remark}

\theoremstyle{remark}
\newtheorem{Case}{Case}

\startlocaldefs

\allowdisplaybreaks
\endlocaldefs

\begin{document}
\begin{frontmatter}

\title{Integral representation with respect to fractional Brownian
motion under a log-H\"older assumption}

\author{\inits{T.}\fnm{Taras}\snm{Shalaiko}}\email{tosha@univ.kiev.ua}
\author{\inits{G.}\fnm{Georgiy}\snm{Shevchenko}\corref
{cor1}}\email
{zhora@univ.kiev.ua}
\cortext[cor1]{Corresponding author.}
\address{Department of Mechanics and Mathematics, Taras Shevchenko
National University of
Kyiv, Volodymyrska 60, 01601 Kyiv, Ukraine}

\markboth{T. Shalaiko, G. Shevchenko}{Integral representation with
respect to fractional Brownian motion}

\begin{abstract}
We show that if a random variable is the final value of an adapted
log-H\"older continuous process, then it can be represented as a
stochastic integral with respect to a fractional Brownian motion with
adapted integrand. In order to establish this representation result, we
extend the definition of the fractional integral.
\end{abstract}


\received{6 September 2015}
\revised{13 September 2015}
\accepted{14 September 2015}
\publishedonline{25 September 2015}
\end{frontmatter}

\section{Introduction}\label{sec1}
This paper can be considered as a continuation of the research started
in \cite{valkeila, mish-shev,shev,shevadapt,vita}. Namely, we
are interested in representations of a random variable in the form of a
stochastic integral
%
%
\begin{equation}
\label{repr} \xi=\int_0^1 \psi(s)\mathrm{d}
B^H(s)
\end{equation}
with respect to a fractional Brownian motion with Hurst parameter $H >
1/2$; the integrand $\psi$ is assumed to be
adapted to the natural filtration generated by~$B^H$ on the interval
$[0,1]$. The motivation comes from financial
mathematics, where capitals of self-financing strategies are given by
stochastic integrals with respect to asset pricing processes.

The main representation results reported in \cite{valkeila,
mish-shev,shev,shevadapt,vita} involve the following assumption about
$\xi$: it is the value at $1$ of an adapted H\"older-continuous
process. Moreover, it was shown in \cite{mish-shev} (see Remark~2.5)
that such an assumption is unavoidable in the methods used in the cited papers.

In this paper, we generalize the existing results by showing the
existence of representation \eqref{repr} under a weaker assumption that
$\xi$ is the value at $1$ of an~adapted \textit{log-H\"older}
continuous process. In order to establish such a representation, we
extend the definition of the fractional integral introduced in~\cite{zahle}.

The paper is organized as follows. Section~\ref{sec2} gives all necessary
prerequisites about the fractional Brownian motion and the definition
of the extended fractional integral. In Section~\ref{sec3}, we prove
the main
representation result. The Appendix contains auxiliary results
concerning the extended fractional integral.
\section{Preliminaries}\label{sec2}
\subsection{General conventions}\label{sec2.1}
Let $(\varOmega,\mathcal F, \mathbb F=\{ \mathcal F_t\}_{t\in[0,1]}, P)$
be a standard stochastic basis. The adaptedness of processes will be
understood with respect to the filtration $\mathbb{F}$.

Throughout the paper, the symbol $C$ will mean a generic constant, the
value of which may change from line to line. When the dependence on
some parameter(s) is important, this will be indicated by superscripts.
A finite random variable of no importance will be denoted by $C(\omega)$.

\subsection{Fractional Brownian motion}\label{sec2.2}

Our main object of interest in this paper is a fractional Brownian
motion (fBm) with Hurst index $H\in(1/2,1)$ on $(\varOmega,\mathcal F,
\mathsf{P})$, that is,\ an $\mathbb{F}$-adapted centered Gaussian
process $B^H=\{B^H(t)\}_{t\geq0}$ with the covariance function
\[
R_H(t,s)=\mathsf{E}\big[B^H(t)B^H(s)\big]=\frac{1} 2
\bigl( t^{2H}+s^{2H}-|t-s|^{2H}\bigr).
\]
By the Kolmogorov--Chentsov theorem, $B^H$ has a continuous
modification, so in what follows, we assume that $B^H$ is continuous.
Moreover, we will need the following statement on the uniform modulus
of continuity of $B^H$ (see, e.g.,~\cite{mishura}).
\begin{thm}\label{unif_modulus}
For an fBm $B^H$, we have
\[
\sup_{t,s\in[0,1]}\frac{|B^H(t)-B^H(s)|}{|t-s|^{H}|\log
(t-s)|^{1/2}} <\infty\, \text{ almost surely.}
\]
\end{thm}
%
\subsection{Small deviations of sum of squared increments of fractional
Brownian motion}\label{sec2.3}
First, we state a small deviation estimate for sum of squares of
Gaussian random variables (see, e.g.,\ \cite{lishao}).\vadjust{\eject}
\begin{lemma}
\label{small}
Let $\{\xi_i\}_{i=1,\ldots,n}$ be jointly Gaussian centered random
variables. Then for all $x$ such that $0<x<\sum_{i=1}^n \mathsf E [\xi^2_i]$,
we have
\begin{gather*}
\mathsf P\left\{\sum_{i=1}^n\xi_i^2 \le x \right\}\leq\exp\biggl
(-\frac{(x-\sum_{i=1}^n \mathsf E [\xi_i^2])^2}{\sum_{i,j=1}^n
(\mathsf E [\xi_i\xi_j])^2} \biggr).
\end{gather*}
\end{lemma}
We will also need the following asymptotics of the covariance of a
fractional Brownian motion. Its proof is given in the Appendix.
\begin{lemma}
\label{fbm-estim}
Let $H\in(1/2,1)$. Set $\Delta B^{H}_i=B^H(i+1)-B^H(i)$ for
$i=0,\ldots
,n-1$. Then the following relation holds as $n\to\infty$:
\begin{gather*}
\sum_{i,j=0}^{n-1} \bigl( \mathsf{E}\big[\Delta
B^H_i \Delta B^H_j\big]
\bigr)^2\sim C_H %
\begin{cases}
n, &H\in(1/2,3/4),\\
n\log n, & H=3/4,\\
n^{4H-2}, & H\in(3/4,1).
\end{cases}
\end{gather*}
%
\end{lemma}

Lemmas~\ref{small} and \ref{fbm-estim} imply the following small
deviation estimate for the sum
of squares of fBm increments.
\begin{lemma}
\label{small-ball-fbm}
Let $B^H=\{B^H_t\}_{t\geq0}$ be an fBm with Hurst index $H>1/2$, $n\ge
2$, and let $\{\Delta B^H_k\}_{k=0,\ldots, n-1}$ be as before. For all
$\alpha\in(0,1)$, we have
\begin{gather*}
\mathsf P\left\{\sum_{k=0}^{n-1}\bigl(\Delta B^H_k\bigr)^2\le
\alpha n \right\}\leq\exp\bigl\{
-C_H(1-\alpha)^2 r(n) \bigr\},
\end{gather*}
where
\[
r(n)= %
\begin{cases}
n, & H \in(1/2,3/4),\\
n/\log n, & H = 3/4,\\
n^{4-4H}, & H \in(3/4,1).
\end{cases}
\]
%
\end{lemma}

\subsection{Extended fractional integral}\label{sec2.4}
To integrate with respect to a fractional Brownian motion, we use the
fractional integral introduced in \cite{zahle}, but the definition is
modified according to our purposes.
For functions $f,g\colon[a,b] \to\mathbb R$ and
$\alpha\in(0,1)$, define the fractional Riemann--Liouville derivatives
(in Weyl form)
\begin{gather*}
\bigl(D_{a+}^{\alpha}f \bigr) (x)=\frac{1}{\varGamma(1-\alpha)}
\Biggl(
\frac
{f(x)}{(x-a)^\alpha}+\alpha\int_{a}^x
\frac{f(x)-f(u)}{(x-u)^{\alpha+1}}\mathrm{d}u \Biggr),
\\
\bigl(D_{b-}^{1-\alpha}g \bigr) (x)=\frac{e^{-i\pi
\alpha}}{\varGamma(\alpha)} \Biggl(
\frac{g(x)}{(b-x)^{1-\alpha}}+(1-\alpha) \int_{x}^b
\frac{g(x)-g(u)}{(u-x)^{2-\alpha}}\mathrm{d}u \Biggr).
\end{gather*}
Then the fractional integral $\int_a^bf(x)\mathrm{d}g(x)$ can be
defined as
%
%
\begin{equation}
\label{ls-int} \int_a^bf(x)\mathrm{d}
g(x)=e^{i\pi\alpha}\int_a^b
\bigl(D_{a+}^{\alpha}f \bigr) (x) \bigl(D_{b-}^{1-\alpha}g_{b-}
\bigr) (x)\mathrm{d}x,
\end{equation}
provided that the last integral is finite. However, in order to have
good properties of this integral (independence of $\alpha$,
additivity), we need to assume more than just the finiteness of the
integral; see the Appendix for details.

Now we turn to the integration with respect to a fractional Brownian
motion. We will restrict our exposition to the interval $[0,1]$, which
suffices for our purposes. Fix a number $\alpha\in(1-H,1/2)$. Note that
from Theorem \ref{unif_modulus} it is easy to derive the following estimate:
\[
\big\vert D^{1-\alpha}_{1-} B^H_{1-}(x)\big\vert\leq C(
\omega)|1-x|^{H+\alpha-1}\big|\log(1-x)\big|^{1/2}.
\]
For some $\mu>1/2$, define the weight
\[
\rho(x) = (1-x)^{H+\alpha-1} \big|\log(1-x)\big|^{\mu}.
\]
Let a function $f\colon[0,1]\to\mathbb R$ be such that
$D_{0+}^{\alpha
}f\in
L^1([0,1],\rho)$; this will be our class of admissible integrands. Then
the \textit{extended fractional integral}
%
%
\begin{equation}
\label{BH-int} \int_{0}^{1}f(x) \mathrm{d}
B^H(x) = e^{i\pi\alpha}\int_0^1
\bigl(D_{0+}^{\alpha
}f \bigr) (x) \bigl(D_{1-}^{1-\alpha}B^H_{1-}
\bigr) (x)\mathrm{d}x
\end{equation}
is well defined (see the Appendix). In particular, it is possible to
take $f$ with $D_{0+}^{\alpha}f\in L^1[0,1]$, and for such integrands,
the definition agrees with the definition of the fractional integral
given in \cite{zahle}; see Remark on p.~340.

Furthermore, it is shown in the Appendix that if $f$ satisfies the
above assumption for a different value of $\alpha$, the value of the
extended fractional integral will be the same. The following estimate
is obvious:
\[
\bigg\vert\int_{0}^{1}f(x) \mathrm{d}B^H(x)\bigg\vert\le
C(\omega) \big\lVert D^{\alpha }_{0+}f\big\rVert_{L^1([0,1],\rho)}.
\]

For each $t\in(0,1)$, we will define the integral $\int_0^t f(x)
\mathrm{d}
B^H(x)$ by a similar formula understanding it in the sense of \cite
{zahle} since $D_{0+}^{\alpha}f\in L^1[0,t]$, $D_{t-}^{1-\alpha
}B^H_{t-}\in L^\infty[0,t]$. Under the additional assumption that
$D_{t+}^{\alpha}f\in L^1([t,1],\rho)$, we can define the integral
$\int_t^1 f(x)\mathrm{d}B^H(x)$ similarly to \eqref{BH-int}, and the
additivity holds:
\[
\int_0^1f(x)\mathrm{d}B^H(x)=\int
_0^t f(x)\mathrm{d}B^H(x)+\int
_{t}^1 f(x)\mathrm{d}B^H(x).
\]
Note that the additivity
\[
\int_0^tf(x)\mathrm{d}B^H(x)=\int
_0^s f(x)\mathrm{d}B^H(x)+\int
_{s}^t f(x)\mathrm{d}B^H(x)
\]
for $s<t$ follows from the results of \cite{zahle}.

Finally, it is worth to add that for $f\in C^\gamma[0,1]$ with $\gamma
>1-H$, the extended fractional integral is well defined since the
derivative $D_{0+}^\alpha f$ is bounded for any \mbox{$\alpha<\gamma$}, and
thus we can take $\alpha\in(1-H,\gamma)$ in the definition. The value
of the integral agrees with the so-called Young integral, which is
given by a limit of integral sums. An important example is $f(x) =
g(B^H(x))$ where $g$ is Lipschitz continuous. In this case, the
following change-of-variable formula holds:
\[
\int_a^b g\bigl(B^H(x)\bigr)\mathrm{d}
B^H(x) = G(b) - G(a),
\]
where $G(x) = \int_0^x g(y)\mathrm{d}y$. The formula appears to be
valid (with
the integral defined in the sense of \cite{zahle}) even for functions
$h$ of locally bounded variation; see \cite{amv}. However, Lipschitz
continuity (even continuous differentiability) will suffice for our purposes.

\section{Main result}\label{sec3}
In this section, for a given $\mathcal F_1$-measurable random variable
$\xi$, we construct an $\mathbb F$-adapted process $\psi=\{\psi(t)\}
_{t\in[0,1]}$ such that
\eqref{repr} holds almost surely
under the following ``log-H\"older'' assumption on $\xi$.

\begin{assum}\label{assump1}
There exists an $\mathbb F$-adapted process $\{Z(t)\}_{t\in[0,1]}$ such
that $Z(1)=\xi$ and, for some $a>1$,
%
%
\begin{equation}
\label{log_hold} \big|Z(1)-Z(t)\big|\leq C(\omega)\big|\log(1-t)\big|^{-a}
\end{equation}
for all $t\in[0,1)$.
\end{assum}

\begin{remark}
Obviously, the process $Z$ satisfies
%
%
\begin{equation}
\label{hold} \big|Z(1)-Z(t)\big|\leq C(\omega) (1-t)^{b}
\end{equation}
for any $b>0$. So \eqref{log_hold} is weaker than \eqref{hold}, which
is the assumption made in \cite{valkeila}.

In \cite{valkeila}, the following example of a random variable not
satisfying \eqref{hold} was given. Assume that $\mathbb F = \{\mathcal
F_t=\sigma(B_s^H,s\in[0,t])\}_{t\in[0,1]}$, and let $\xi= \int
_{1/2}^1 g(t){dW_t}$, where $g(t) = (1-t)^{-1/2}\vert\log(1-t)\vert^{-1}$,
and $W$ is a Wiener process such that its natural filtration coincides
with $\mathbb F$. Using the same argument as in \cite{valkeila}, it is
possible to show that $\xi$ does not satisfy \eqref{log_hold} as well.
However, if we take the same construction with $g(t) =
(1-t)^{-1/2}\vert\log(1-t)\vert^{-d}$, $d>1$, then the corresponding
random variable
satisfies \eqref{log_hold}, but not \eqref{hold}.
\end{remark}

Next, we state a helpful lemma from \cite{valkeila}, used in our
construction of $\psi$.
\begin{lemma}
\label{lemma-inft}
There exists an $\mathbb{F}$-adapted process $\phi$ on $[0,1]$ such
that for any $t>0$, $D^\alpha_{0+}\phi\in L^1[0,t]$, so that the integral
$\int_0^t \phi(s)\mathrm{d}B^H(s)$
is defined as the fractional integral, and
\begin{equation*}
\lim_{t\to1-} \int_0^t \phi(s)
\mathrm{d}B^H(s) = +\infty
\end{equation*}
almost surely.
\end{lemma}
We can now proceed with the main result.
\begin{thm}
Let $\xi$ satisfy Assumption~\ref{assump1}. Then there exists an
$\mathbb
F$-adapted process $\psi=\{\psi(t)\}_{t\in[0,1]}$ such that \eqref
{repr} holds with the integral defined in the extended fractional sense.
\end{thm}
\begin{proof}
The proof is divided into three parts.

\textbf{Construction of }$\mathbf\psi$. Let $\kappa\in(2,2^a)$. Put
$t_n = 1-e^{-\kappa^{n/a}},\ n\ge1$, and let $\Delta_n = t_{n+1}-t_n$.
It is easy to see that
\begin{gather}
(1-t_n)\le C\Delta_n. \label{t_n-ineq}
\end{gather}
Denote for brevity $\xi_n = Z(t_n)$. Then, by Assumption~1, $\vert\xi
_n -\xi\vert\le C(\omega) \kappa^{-n}$, so that $|\xi_n - \xi|\le
2^{-n}$ for
all $n$ large enough, say, for $n\ge N(\omega)$. In particular, we have
%
%
\begin{equation}
\label{xi-incr} |\xi_n - \xi_{n-1}| \le2^{-n+2}
\end{equation}
for all $n\ge N(\omega)+1$.

The integrand $\psi$ is constructed in an inductive way between the
points $\{t_n,n\ge1\}$. Set first $\psi(t)=0$, $t\in[0,t_1]$. Assuming
that $\psi(t)$ is defined on $[0,t_n)$, let $V(t)=\int_0^t \psi
(s)\mathrm{d}
B^H(s), t\in[0,t_n]$. The construction of the integrand $[t_n,t_{n+1})$
depends on whether $V(t_n)=\xi_{n-1}$ or not.

%
\begin{Case}\label{case1}
$V(t_n)\neq\xi_{n-1}$. In this case, thanks to
Lemma \ref{lemma-inft}, there exists an adapted process $\{\phi
(t),t\in
[t_n,t_{n+1}]\}$ such that $\int_{t_n}^{t}\phi(s)\mathrm{d}B^H(s)\to
+\infty$
as $t\to t_{n+1}-$. Define the stopping time
\begin{gather*}
\tau_n=\inf\Biggl\{t\geq t_n: \int_{t_n}^{t}
\phi(s)\mathrm{d}B^H(s)\geq|\xi_n-V_{t_n}| \Biggr\}
\end{gather*}
and the process
\begin{gather*}
\psi(t)=\phi(t) \mathsf{sign}\big(\xi_n-V(t_n)\big)\mathbb{I}_{[t_n,\tau_n]}(t), \quad t\in[t_n,t_{n+1}).
\end{gather*}
It is obvious that $\int_{t_n}^{t_{n+1}}\psi(s)\mathrm{d}B^H(s)=\xi
_n-V(t_n)$
and $V(t_{n+1})=\xi_n$.
\end{Case}

%
\begin{Case}\label{case2}
$V(t_n)=\xi_{n-1}$. We consider the uniform
partition $s_{n,k} = t_n + k\delta_n$, $k=1,\ldots,n$ of
$[t_n,t_{n+1}]$ with mesh $\delta_n=\Delta_n n^{-1}$ and auxiliary function
\begin{gather*}
\overline{\phi}(t)=a_n\sum_{k=0}^{n-1}
\bigl( B^H(t)-B^H(s_{n,k})\bigr)\mathbb{I}
_{[s_{n,k},s_{n,k+1})}(t),
\end{gather*}
with $a_n=2^{-n+3}\delta_n^{-2H}n^{-1}$. Notice that by the
change-of-variable formula,
%
%
\begin{equation}
\label{int_n} \int_{t_n}^{t_{n+1}} \overline{\phi}(t) \mathrm{d}
B^H(t) = a_n \sum_{k=0}^{n-1}
\bigl(B^H(s_{n,k+1})-B^H(s_{n,k})
\bigr)^2.
\end{equation}

Define the stopping time
\begin{gather*}
\sigma_n=\inf\Biggl\{t\geq t_n: \int
_{t_n}^t \overline{\phi}(s)\mathrm{d}B^H(s)
\geq|\xi_n-\xi_{n-1}| \Biggr\}\wedge t_{n+1}
\end{gather*}
and set
\begin{gather*}
\psi(t)=\mathsf{sign}(\xi_n-\xi_{n-1})\overline{\phi}(t) \mathbb{I}_{[t_n,\sigma
_n]}(t), \quad t\in[t_n,t_{n+1}). 
\end{gather*}
\end{Case}

\textbf{The construction approaches} $\mathbf{\xi}$. Our aim now is to
prove that $V(t_{n}) = \xi_{n-1}$ for all $n$ large enough. By
construction it suffices to show that Case~\ref{case2} happens for all $n$ large
enough. Equivalently, we need to show that $\sigma_n <t_{n+1}$ for all
$n$ large enough. In view of \eqref{int_n}, the latter inequality
holds if
\[
a_n \sum_{k=0}^{n-1}
\bigl(B^H(s_{n,k+1})-B^H(s_{n,k})
\bigr)^2> \vert\xi_n - \xi_{n-1}\vert.
\]
Thus, in view of the Borell--Cantelli lemma and inequality \eqref
{xi-incr}, it suffices to verify the convergence of the series
\begin{gather*}
\sum_{n\geq1} \mathsf P\left\{a_n\sum _{k=1}^{n-1} \bigl
(B^H(s_{n,k+1})-B^H(s_{n,k}) \bigr)^2<2^{-n+2} \right\}
\end{gather*}
or, equivalently, that
\begin{gather*}
\sum_{n\geq1} \mathsf P\left\{\sum_{k=1}^{n-1} \bigl(\delta
_n^{-H}\bigl(B^H_{s_{n,k+1}}-B^H_{s_{n,k}} \bigr) \bigr)^2<n/2
\right\}<\infty.
\end{gather*}
The latter follows from Lemma \ref{small-ball-fbm} through the
self-similarity and stationarity of increments of an fBm. Thus, for all
$n$ large enough, say, for $n\ge N_2(\omega)$, $V(t_n) = \xi_{n-1}$.

\textbf{Integrability of} $\mathbf\psi$. It is easy to see that
$\psi$
is integrable w.r.t.\ $B^H$ on any interval $[0,t_N]$. It remains to
verify that the integral $\int_{t_N}^1 \psi(s) \mathrm{d}B^H(s)$ is well
defined and vanishes as $N\to\infty$ (note that we did not establish
the continuity of the integral as a function of the upper limit). For
some $\mu>1/2$ (which will be specified later), define
\[
\rho(x) = (1-x)^{H+\alpha-1} \big|\log(1-x)\big|^{\mu}.
\]
Clearly, it suffices to show that $\lVert D^{\alpha}_{t_N+}\psi
\rVert_{L^1([t_N,1],\rho)}\to0$, $N\to\infty$.
Let $N\ge N_2(\omega)$. Write
%
%
\begin{align*}
\int_{t_N}^1\big|\bigl(D^{\alpha}_{t_N+}\psi\bigr) (s)\big|\rho(s)\mathrm
{d}s &=
\sum_{n=N}^\infty\int_{t_n}^{t_{n+1}}\big|\bigl(D^{\alpha}_{t_N+}\psi
\bigr) (s)\big|\rho(s)\mathrm{d}s\\
&\le C\sum_{n=N}^\infty\Delta_n^{H+\alpha-1}|\log\Delta_n|^{\mu
}\int_{t_n}^{t_{n+1}}\big|\bigl(D_{t_N+}^{\alpha}\psi\bigr)
(s)\big|\mathrm{d}s,
\end{align*}
%
%
where we have used \eqref{t_n-ineq} and the fact that $\rho$ is
decreasing in a left neighborhood of $1$.

Now we estimate
\begin{align*}
&\int_{t_n}^{t_{n+1}}\big|\bigl(D_{t_N+}^{\alpha}\psi\bigr)
(s)\big|\mathrm{d}s
\leq\int_{t_n}^{t_{n+1}} \Biggl(\frac{|\psi(s)|}{(s-t_N)^{\alpha
}}+\int_{t_N}^s \frac{|\psi(s)-\psi(u)|}{|s-u|^{1+\alpha}}\mathrm{d}u
\Biggr) \mathrm{d}s\\
&\quad\le C(\omega)a_n \Delta_n^{1-\alpha}\delta^{H}_n\big|\log
(\delta_n)\big|^{1/2}+\int_{t_N}^{t_{n+1}}\int_{t_n}^s\frac{|\psi
(s)-\psi(u)|}{|s-u|^{1+\alpha}}\mathrm{d}u\, \mathrm{d}s,
\end{align*}
where we have used Theorem~\ref{unif_modulus} to estimate $\psi$.

Consider the second term. It equals
\begin{align*}
\sum_{k=1}^n \int_{s_{n,k-1}}^{s_{n,k}}
\Biggl(\int_{t_N}^{t_n}+\int_{t_n}^{s_{n,k-1}}+
\int_{s_{n,k-1}}^{s} \Biggr)\frac{|\psi(s)-\psi
(u)|}{|s-u|^{1+\alpha}}\mathrm{d}u\, \mathrm{d}
s=:I_1+I_2+I_3.
\end{align*}
Start with $I_1$, observing that $\psi$ vanishes on $(\sigma_n,t_{n+1}]$:
\begin{align*}
I_1&\leq\int_{t_n}^{t_{n+1}}\sum_{j=N}^n \int_{t_{j-1}}^{t_j}\frac
{|\psi(s)|+|\psi(u)|}{|s-u|^{1+\alpha}}\mathrm{d}u\, \mathrm{d}s\\
&\leq C(\omega) a_n \delta_n^{H}|\log\delta_n|^{1/2}\int
_{t_n}^{t_{n+1}}(s-t_n)^{-\alpha}\mathrm{d}s\\
&\quad+\sum_{j=N}^{n-1} a_j \delta^{H}_j|\log\delta_j|^{1/2}\int
_{t_n}^{t_{n+1}}(s-t_{j+1})^{-\alpha}\mathrm{d}s\\
&\leq C(\omega) \Biggl(a_n \Delta^{1-\alpha}_n\delta^H_n|\log
\delta_n|^{1/2}+\sum_{j=N}^{n-1}a_j\delta^H_j |\log\delta
_j|^{1/2}\Delta_n^{1-\alpha} \Biggr).
\end{align*}
Proceed with the second term:
\begin{align*}
I_2&\leq C(\omega)a_n\delta^H_n|\log\delta_n|^{1/2} \sum
_{k=1}^n\int_{s_{n,k-1}}^{s_{n,k}}\int
_{t_n}^{s_{n,k-1}}|s-u|^{-1-\alpha}\mathrm{d}u\, \mathrm{d}s\\
&\le C(\omega)a_n\delta^H_n|\log\delta_n|^{1/2}\sum_{k=1}^n\int
_{s_{n,k-1}}^{s_{n,k}} (s-s_{n,k-1})^{-\alpha}\mathrm{d}s\\
&\le C(\omega)a_n n \delta_n^{H+1-\alpha}|\log\delta
_n|^{1/2}=C(\omega)a_n\Delta_n\delta_n^{H-\alpha}|\log\delta_n|^{1/2}.
\end{align*}
Finally, assuming that $\sigma_n\in[s_{n,l-1},s_{n,l})$, we have
\begin{align*}
I_3&\leq C(\omega)\sum_{k=1}^{l-1}\int_{s_{n,k-1}}^{s_{n,k}}\int_{s_{n,k-1}}^s a_n\frac{(s-u)^{H}|\log(s-u)|^{1/2}}{(s-u)^{1+\alpha}}\mathrm{d}u\, \mathrm{d}s\\
&\quad+ \int_{s_{n,l-1}}^{\sigma_n}\int_{s_{n,l-1}}^s\frac{|\psi(s)-\psi(u)|}{|s-u|^{1+\alpha}}\mathrm{d}u\, \mathrm{d}s+\int_{\sigma_n}^{s_{n,l}}\int_{s_{n,l-1}}^{\sigma_n}\frac{|\psi(s)-\psi(u)|}{|s-u|^{1+\alpha}}\mathrm{d}u\, \mathrm{d}s\\
&\le C(\omega)a_n\sum_{k=1}^n\int_{s_{n,k-1}}^{s_{n,k}}(s-s_{n,k-1})^{H-\alpha}\big|\log(s-s_{n,k-1})\big|^{1/2}\mathrm{d}s\\
&\quad+C(\omega) a_n \delta_n^H |\log\delta_n|^{1/2}\int_{\sigma_n}^{s_{n,l}}\int_{s_{n,l-1}}^{\sigma_n}\frac{1}{|s-u|^{1+\alpha}}\mathrm{d}u\, \mathrm{d}s\\
&\leq C(\omega)a_n n\delta_n^{H+1-\alpha}|\log\delta_n|^{1/2}n\delta_n^{1-\alpha}= C(\omega)a_n\Delta_n \delta^{H-\alpha}_n|\log\delta_n|^{1/2}.
\end{align*}

Gathering all estimates, we get
\begin{align*}
&\int_{t_N}^1 \big|D^\alpha_{t_N+}(\psi) (s)\big|\rho(s)\mathrm{d}s \leq
C(\omega
)\sum_{n=N}^\infty\Biggl(a_n\Delta^{H}_n\delta^{H}_n|\log\delta
_n|^{1/2}|\log\Delta_n|^{\mu}\\
&\quad+a_n\Delta^{H+\alpha}_n\delta^{H-\alpha}_n|\log\delta
_n|^{1/2}|\log\Delta_n|^{\mu}+\Delta_n^H|\log\Delta_n|^{\mu}\sum
_{j=N}^{n-1}a_j\delta_j^H|\log\delta_j|^{1/2}\Biggr).
\end{align*}
Consider the second sum. After changing the order of summation, we get
\begin{align*}
&\sum_{n=N}^\infty\sum_{j=N}^{n-1}a_j\delta_j^H|\log\delta
_j|^{1/2}\Delta^H_n|\log\Delta_n|^\mu= \sum_{j=N}^\infty a_j\delta
_j^H\sum_{n = j+1}^\infty\Delta_n^H |\log\Delta_n|^\mu\\[-1pt]
&\quad\le C\sum_{j=N}^\infty a_j\delta_j^H \Delta_j^H |\log\Delta
_j|^\mu\sum_{n = j+1}^\infty e^{\kappa^{j/a} - \kappa^{n/a}}\kappa
^{\mu(n-j)/a}\\[-1pt]
&\quad\le C\sum_{j=N}^\infty a_j\delta_j^H \Delta_j^H|\log\Delta
_j|^\mu\\[-1pt]
&\quad\le C\sum_{j=N}^\infty a_j\delta_j^H \Delta_j^H |\log\Delta
_j|^\mu\sum_{n = j+1}^\infty e^{- \kappa^{(n-j)n/a}}\kappa^{\mu
(n-j)/a}\\[-1pt]
&\quad\le C \sum_{j=N}^\infty a_j\delta_j^H \Delta_j^H |\log\Delta
_j|^\mu.
\end{align*}
Consequently, noting that $\Delta_n^H\delta_n^H\leq\Delta
_n^{H+\alpha
}\delta_n^{H-\alpha}$, we have
\begin{align*}
&\int_{t_N}^1 \big|D^\alpha_{t_N+}(\psi) (s)\big|\rho(s)\mathrm{d}s \leq
C(\omega
)\sum_{n=N}^\infty a_n\Delta^{H+\alpha}_n\delta^{H-\alpha}_n|\log
\delta_n|^{1/2}|\log\Delta_n|^{\mu}\\[-1pt]
&\quad\le\sum_{n=N}^\infty2^{-n}n^{-1}\Delta_n^{H+\alpha}\delta
_n^{-H-\alpha}|\log\delta_n|^{\mu+1/2}\\[-1pt]
&\quad\le\sum_{n=N}^\infty2^{-n}n^{H+\alpha-1}\bigl(|\log\Delta
_n|^{\mu+1/2} + n^{\mu+1/2}\bigr)\\[-1pt]
&\quad\le\sum_{n=N}^\infty2^{-n}n^{H+\alpha-1}\bigl(C + \kappa
^{(\mu+1/2)n/a} + n^{\mu+1/2}\bigr).
\end{align*}
Now, in order for the last series to converge, it is sufficient that
$(\mu+1/2)/a < \log2/\log\kappa$ or, equivalently, $\mu< a \log
2/\log\kappa-1/2$. Since $a \log2/\log\kappa> 1$ by our choice, it
is possible to take some $\mu>1/2$ satisfying this requirement, thus
finishing the proof.
\end{proof}

\appendix
\section{Appendix. Properties of the extended fractional integral}

In this section, we establish important properties of the extended
fractional integral defined by \eqref{ls-int}. We will restrict
ourselves to the case $[a,b] = [0,1]$. Assume that the function
$g\colon
[0,1]\to\mathbb R$ satisfies the following regularity requirement:
there exist
$\lambda\in(0,1)$ and $\nu>0$ such that
%
%
\begin{equation}
\label{g-cont} \big\vert g(x)-g(y)\big\vert\leq C \vert x-y\vert^{\lambda}
\big|\log|x-y|
\big|^{\nu}
\end{equation}
for all $x,y\in[0,1]$.
We first state some properties of $g$ that follow from this
regularity.\vadjust{\eject}
\begin{lemma}
\label{lemma1}
Let $g\colon[0,1]\to\mathbb R$ satisfy \eqref{g-cont}. Then for any
$\beta
\in(0,\lambda)$,
\[
\big\vert\bigl(D^{\beta}_{1-}g_{1-}\bigr) (x)- \bigl(D^{\beta
}_{1-}g_{1-}\bigr) (y)\big\vert \leq
C_\beta|x-y|^{\lambda-\beta} \big|\log\vert x-y\vert \big|^{\nu}
\]
for all $x,y\in[0,1]$.
\end{lemma}
The proof is similar to that of Lemma~13.1 in \cite{samko} and thus omitted.

Now, for some $\mu>\nu$, define
%
%
\begin{equation}
\label{rho} \rho(x) = |1-x|^{\lambda-\beta}\big|\log(1-x)\big|^{\mu}, \quad x \in[0,1].
\end{equation}
\begin{lemma}
\label{lemma2}
Let $g\colon[0,1]\to\mathbb R$ satisfy \eqref{g-cont}, and $\psi\in
C^\infty
[-1,0]\to\mathbb R$ be a positive function with $\int_{-1}^0 \psi
(x)\mathrm{d}x=1$.
Define $\psi_N(x)=N\psi(Nx)$,
\[
g_N(x)=(g_{1-}\ast\psi_N) (x)= \int
_{x}^{(x+1/N)\wedge1}g_{1-}(y)\psi
_N(x-y)\mathrm{d} y,\quad N\ge1.
\]
Then for any $\mu>\nu$, $\beta\in(0,\lambda)$, $x\in[0,1]$, and all
$N\ge1$ large enough, we have
\begin{gather}
\big\vert\bigl(D^{\beta}_{1-}g_N\bigr) (x)- \bigl(D^{\beta
}_{1-}g_{1-}\bigr) (x)\big\vert\leq
C_\beta(\log N)^{\nu-\mu}\rho(x).
\end{gather}
\end{lemma}
\begin{proof}
Note that $D_{1-}^{\beta}g_N= \mathbf1_{[0,1]}(D_{1-}^{\beta
}g_{1-})\ast\psi_N$. For brevity, set $h(x)=\break(D_{1-}^{\beta
}g_{1-})(x)$, $x\in[0,1]$. First, suppose that $x\in[0,1-1/N]$. Thanks
to Lemma~\ref{lemma1},
\begin{align*}
\big\vert\bigl(D^{\beta}_{1-}g_N\bigr) (x)-\bigl(D^{\beta
}_{1-}g_{1-}\bigr) (x)\big\vert &\leq C \int_{x}^{x+1/N}\big|h(x)-h(y)\big|N\psi
\bigl(N(x-y)\bigr)\mathrm{d}y\\
&\le C \int_{-1}^0\big|h(x-y/N)-h(y)\big|\psi(y)\mathrm{d}y\\
&\le C_\beta N^{\beta-\lambda}\int_{-1}^0|y|^{\lambda-\beta}\big(N -
\log|y|\big)^\nu\psi(y)\mathrm{d}y\\
&\leq C N ^{\beta-\lambda}(\log N)^\nu\le C (\log N)^{\nu-\mu}\rho(x),
\end{align*}
where the last inequality holds for all $N$ large enough since $\rho
(x)$ is decreasing for $x$ close to $1$.

Now consider the case $x\in[1-1/N,1]$. Using Lemma \eqref{lemma1}, we have
\begin{align*}
&\big\vert\bigl(D^{\beta}_{1-}g_N\bigr) (x)-\bigl(D^{\beta
}_{1-}g_{1-}\bigr) (x)\big\vert\le\big\vert\bigl(D^{\beta}_{1-}g_N\bigr)
(x)\big\vert+\big\vert\bigl(D^{\beta}_{1-}g_{1-}\bigr) (x)\big\vert\\
&\quad\le\int_{x}^1 N\psi\bigl(N(x-y)\bigr)\big|h(1)-h(y)\big|\mathrm{d}
y+\big|h(1)-h(x)\big|\\
&\quad\le C |1-x|^{\lambda-\beta}\big|\log(1-x)\big|^{\nu}\leq C\rho(x)
\big|\log(1-x)\big|^{\nu-\mu}\\
&\quad\le C\rho(x) (\log N)^{\nu-\mu}
\end{align*}
since $x>1-1/N$ and $\nu<\mu$. This finishes the proof.
\end{proof}

\begin{thm}
Let $g\colon[0,1]\to\mathbb R$ satisfy \eqref{g-cont}, $\alpha\in
(1-\lambda
, 1)$, and let $\rho$ be given by \eqref{rho} with $\beta=1-\alpha$.
Assume further that $f\colon[0,1]\to\mathbb R$ is such that
$D^{\alpha}_{0+}
f\in L^1([0,1], \rho)$. Then the extended fractional integral
\[
\int_{0}^{1} f(x)\mathrm{d}g(x)=e^{i\pi\alpha}\int
_0^1\bigl(D^{\alpha
}_{1-}f
\bigr) (x) \bigl(D^{1-\alpha}_{1-}g_{1-}\bigr) (x)\mathrm{d}x
\]
is well defined. Moreover, if $f$ and $g$ satisfy the above assumptions
for different $\alpha$, the value of the integral is preserved.
Additionally, if $D^{\alpha}_{t+} f\in L^1([t,1], \rho)$ for some
$t\in
(0,1)$, then
\[
\int_0^1f(x)\mathrm{d}g(x)=\int_0^t
f(x)\mathrm{d}g(x)+\int_{t}^1 f(x)\mathrm{d}g(x),
\]
where the first integral is understood in the sense of \emph{\cite
{zahle}} $(D^{\alpha}_{0+} f\in L^1[0,t]$,\break$D^{1-\alpha}_{t-}
g_{t-}\in
L^\infty[0,t])$, and the second in the above sense.
\end{thm}
\begin{proof}
From Lemma~\ref{lemma1} we have
\[
\big\vert\bigl(D_{1-}^{1-\alpha}g\bigr) (x)\big\vert \le C(1-x)^{\lambda+
\alpha- 1}
\big\vert\log(1-x)\big\vert^{\nu}\le C\rho(x),
\]
whence the finiteness of the integral follows. Defining $g_N,N\geq1$,
as in Lemma~\ref{lemma2}, we have
%
%
\begin{align}
&e^{i\pi\alpha}\int_0^1\bigl(D^{\alpha}_{1-}f\bigr) (x) \bigl
(D^{1-\alpha}_{1-}g_{1-}\bigr) (x)\mathrm{d}x\nonumber\\
&\quad= e^{i\pi\alpha}\lim_{N\to\infty} \int_0^1\bigl(D^{\alpha
}_{1-}f\bigr) (x) \bigl(D^{1-\alpha}_{1-}g_{N}\bigr) (x)\mathrm
{d}x \label
{int-conv}
\end{align}
%
%
in view of the dominated convergence theorem.

Since $g_N\in C^\infty [0,1]\to \mathbb R$, it follows from the properties of
fractional integrals and derivatives (see \cite{samko,zahle}) that
$
(D^{1-\alpha}_{1-} g_N)(x) = (I^\alpha_{1-} g'_N)(x),
$
where
\[
\bigl(I^\alpha_{1-} h\bigr) (x) = \frac{e^{-i\pi\alpha
}}{\varGamma(\alpha)}\int
_x^1 (y-x)^{\alpha- 1} h(y)\mathrm{d}y
\]
is the right-sided fractional Riemann--Liouville integral of order
$\alpha$.
Note that
\begin{align*}
&\int_0^1 \int_x^1\big\vert\bigl(D^{\alpha}_{1-}f\bigr) (x)\big\vert
(y-x)^{\alpha- 1}\big\vert g_N'(y)\big\vert\mathrm{d}y\, \mathrm{d}x\\
&\quad\le C_N \int_0^1 \big\vert\bigl(D^{\alpha}_{1-}f\bigr) (x)\big\vert
(1-x)^{\alpha}\mathrm{d}x\le C_N \int_0^1 \big\vert\bigl(D^{\alpha
}_{1-}f\bigr) (x)\big\vert \rho(x)\mathrm{d}x<\infty
\end{align*}
thanks to our assumptions.

Hence, by the Fubini theorem,
\begin{align*}
&e^{i\pi\alpha}\int_0^1\bigl(D^{\alpha}_{1-}f\bigr) (x) \bigl
(D^{1-\alpha}_{1-}g_{N}\bigr) (x)\mathrm{d}x\\
&\quad= \frac{1}{\varGamma(\alpha)}\int_0^1 \int_0^y
(y-x)^{1-\alpha}\bigl(D^{\alpha}_{1-}f\bigr) (x)\mathrm{d}x\,
g_N'(y)\mathrm{d}y\\
&\quad= \int_0^1 \bigl(I^\alpha_{0+}D^\alpha_{0+} f\bigr)
(y)g_N'(y)\mathrm{d}y,
\end{align*}
where
\[
\bigl(I^\alpha_{0+} h\bigr) (x) = \frac{1}{\varGamma(\alpha)}\int
_0^x (x-y)^{\alpha-
1} h(y)\mathrm{d}y
\]
is the left-sided fractional Riemann--Liouville integral of order
$\alpha$.
Note that $D^\alpha_{0+} f\in L^1[0,t]$ for any $t\in(0,1)$, so by
\cite
[Theorem~2.4]{samko}, for almost all $y\in[0,t]$, we have
$(I^\alpha_{0+} D^\alpha_{0+} f)(y) = f(y)$. Therefore, this equality
holds for almost all $y\in[0,1]$, and
\begin{gather*}
e^{i\pi\alpha}\int_0^1\bigl(D^{\alpha}_{1-}f
\bigr) (x) \bigl(D^{1-\alpha}_{1-}g_{N}\bigr) (x)\mathrm{d} x=
\int_0^1 f(y)g_N'(y)
\mathrm{d}y.
\end{gather*}
Taking into account \eqref{int-conv}, we get that, indeed, the extended
fractional integral $\int_0^1 f(s)$ does not depend on $\alpha$.

Concerning additivity, under the assumption $D^{\alpha}_{t+} f\in
L^1([t,1], \rho)$, repeating the previous argument, we get
\[
\int_t^1 f(x) \mathrm{d}g(x) = \lim
_{N\to\infty} \int_t^1
f(y)g_N'(y)\mathrm{d}y.
\]
On the other hand, from the proof of \cite[Theorem~2.5]{zahle} we have
\[
\int_0^t f(x) \mathrm{d}g(x) = \lim
_{N\to\infty} \int_0^t
f(y)g_N'(y)\mathrm{d}y.
\]
Therefore,
\[
\int_0^t f(x) \mathrm{d}g(x) + \int
_t^1 f(x) \mathrm{d}g(x) = \lim_{N\to\infty}
\int_0^1 f(y)g_N'(y)
\mathrm{d}y = \int_0^1 f(x) \mathrm{d}g(x),
\]
as required.
\end{proof}

\section{Appendix. Proof of Lemma~2}
\begin{proof}
We have $\mathsf E [\Delta B^H_i \Delta B^H_j] = \rho_H(\vert i-j\vert
)$ with
\[
\rho_H(m) = \frac{1}2 \bigl((m+1)^{2H} +
\vert m-1\vert^{2H} - 2 m^{2H} \bigr), \quad m\ge0.
\]
It is easy to see that
%
%
\begin{equation}
\label{rhoHasym} \rho_H(m)\sim H(2H-1)m^{2H-2}, \quad m\to\infty.
\end{equation}
The double sum in the question can be transformed as follows:
\begin{align*}
S^H_n:=\sum_{i,j=0}^{n-1}\bigl( \mathsf{E}\big[\Delta B^H_i \Delta
B^H_j\big]\bigr
)^2 &= \sum_{i,j=0}^{n-1}\rho_H\big(\vert i-j\vert\big)^2 = \vert m=i-j\vert
\\
&= \sum_{m=1-n}^{n-1} \big(n-\vert m\vert\big)\rho_H\big(|m|\big)^2.
\end{align*}

Let $H\in(1/2,3/4)$. Then
\begin{gather*}
\frac{S_n^H}n = \sum_{m=1-n}^{n-1}
\biggl(1-\frac{\vert m\vert}{n} \biggr)\rho_H\big(|m|\big)^2.
\end{gather*}
In view of \eqref{rhoHasym}, the series $\sum_{m = -\infty}^{+\infty}
\rho_H(|m|)^2$ converges, so by the dominated convergence theorem,
\[
\frac{S_n^H}n\to\sum_{m = -\infty}^{+\infty}
\rho_H(|m|)^2, \quad n\to\infty.
\]

Now let $H\in[3/4,1)$. Define $a_n = n\log n$ if $H=3/4$, $a_n =
n^{4H-2}$ if $H\in(3/4,1)$. Applying the Stolz--C\`esaro theorem, we have
\[
\lim_{n\to\infty}\frac{S_n^H}{a_n} = \lim_{n\to\infty}
\frac{S_{n+1}^H -
S_n^H}{a_{n+1}-a_n} = \lim_{n\to\infty}\frac{1}{a_{n+1}-a_n}\sum
_{m=-n}^n \rho_H\big(|m|\big)^2,
\]
provided that the latter limit exists. Using the Stolz--C\`esaro
theorem once more, we have
\[
\lim_{n\to\infty}\frac{1}{a_{n+1}-a_n}\sum
_{m=-n}^n \rho_H\big(|m|\big)^2 = 2
\lim_{n\to\infty}\frac{\rho_H(n)^2}{a_{n+1}+a_{n-1}- 2a_n},
\]
provided that the latter limit exists.
It is easy to see that
\[
a_{n+1}+a_{n-1}- 2a_n\sim%
\begin{cases}
n^{-1}, \quad H = 3/4,\\
(4H-2)(4H-3)n^{4H-4}, \quad H\in(3/4,1),
\end{cases}
\]
as $n\to\infty$. Thus, we get the required statement thanks to \eqref
{rhoHasym}.
\end{proof}

\bibliographystyle{vmsta-mathphys}
%

%
\end{document}